\documentclass[12pt]{article}

\usepackage{latexsym,amssymb,amsmath,bm}

\newcommand{\C}{\mathbb C}
\newcommand{\Z}{\mathbb Z}
\newcommand{\N}{\mathbb N}
\newcommand{\F}{\mathbb F}
\newcommand{\E}{\mathbb E}
\newcommand{\Q}{\mathbb Q}
\newcommand{\R}{\mathbb R}
\newcommand{\PP}{\mathbb P}
\newcommand{\HH}{\mathbb H}
\newcommand{\sma}{\left(\begin{array}}
\newcommand{\fma}{\end{array}\right)}

\newcommand{\ep}{\epsilon}

\newtheorem{lem}{Lemma}[section]
\newtheorem{defn}[lem]{Definition}

\newtheorem{co}[lem]{Corollary}
\newtheorem{thm}[lem]{Theorem}
\newtheorem{prop}[lem]{Proposition}

\newenvironment{proof}{\textbf{Proof.}}{\newline\hspace*{\fill}{$\Box$}\\}

\begin{document}
\title{Properties of linear groups with restricted
unipotent elements}
\author{J.\,O.\,Button}

\newcommand{\Address}{{
  \bigskip
  \footnotesize

\textsc{Selwyn College, University of Cambridge,
Cambridge CB3 9DQ, UK}\par\nopagebreak
  \textit{E-mail address}: \texttt{j.o.button@dpmms.cam.ac.uk}
}}

\date{}
\maketitle
\begin{abstract}
We consider linear groups which do not contain unipotent elements
of infinite order, which includes all linear groups 
in positive characteristic, and show that this class of groups
has good properties which resemble those held by groups of  non positive
curvature and which do not hold for arbitrary 
characteristic zero linear groups.
In particular if such a linear group is
finitely generated then centralisers virtually split and all finitely
generated abelian subgroups are undistorted. If further the group
is virtually torsion free (which always holds in characteristic zero)
then we have a strong property on small subgroups: any subgroup either
contains a non abelian free group or is finitely generated and 
virtually abelian, hence also undistorted.
We present applications, including that the mapping
class group of a surface having genus at least 3
has no faithful linear representation which is complex
unitary or over any field of positive characteristic.
\end{abstract}

\section{Introduction}

Knowing that an abstract group $G$ is actually linear (here meaning
that $G$ embeds in the general linear group $GL(d,\F)$ for some
$d$ and field $\F$ of arbitrary characteristic) is, or at least
should be, an indication that $G$ is a well behaved group.
For instance if $G$ is also finitely generated then it is residually
finite (but this is not true of linear groups in general), whereas
the Tits alternative states that if $\F$ has characteristic zero, or if
again $G$ is finitely generated (but not necessarily for arbitrary
linear groups in positive characteristic) then $G$ either contains a
non abelian free subgroup or is a virtually solvable group. In fact
if $G$ is finitely generated and linear in any characteristic then
all subgroups $S$ of $G$, whether finitely generated or not,
also satisfy this alternative in that $S$ either contains the non abelian
free group $F_2$ or is virtually solvable (this is outlined in Section 4).

Now on sticking to the case where $G$ is both linear and finitely
generated, we can try and go further with
those subgroups $S$ of $G$ that do not contain $F_2$ by asking whether 
$S$ has to be finitely generated and also
whether $S$ is always virtually polycyclic, virtually nilpotent or
virtually abelian. Moreover we can think about the geometry of the embedding
in the cases where $S$ does turn out to be finitely generated and ask
if $S$ is undistorted in $G$. For positive results in a different
context, suppose we take any word
hyperbolic group $\Gamma$. In this case $\Gamma$ will be finitely generated,
indeed finitely presented, but it need not be linear over any field.
However we have a very strong property held by any ``small'' subgroup
of $\Gamma$, namely that if $S\leq \Gamma$ is a subgroup not containing
$F_2$ then $S$ is virtually cyclic and is
undistorted in $\Gamma$.

In this paper we will
certainly not want to eliminate groups $G$ containing $\Z^2$, so the correct
analogy here
for the ``strong behaviour of small subgroups'' seen in the
word hyperbolic case is that
every subgroup of $G$ either contains a non abelian free subgroup or
is virtually abelian, finitely generated and undistorted in $G$.   
However let us consider the following
examples, all of which are finitely generated linear groups 
or subgroups thereof.\\
\hfill\\
{\bf  Example 1.1}: The wreath product $\Z\,\wr\,\Z$, which can be
thought of as a semidirect product
$\left(\bigoplus_{i=-\infty}^{i=+\infty}\Z\right)\rtimes\Z$, is 
finitely generated
and a subgroup of $GL(2,\R)$ via the matrices
$\sma{cc}\pi&0\\0&1\fma$ and $\sma{cc}1&1\\0&1\fma$
which is credited to Philip Hall. In fact here we only use the
fact that $\pi$ is transcendental, so that this group is actually
linear in two dimensions over the field $\Q(t)$ for $t$ an arbitrary element
which is transcendental over $\Q$. Clearly this group is solvable, even
metabelian, but not virtually abelian and it contains the subgroup
$\bigoplus_{i=-\infty}^{i=+\infty}\Z$ which is abelian but not finitely
generated.\\
\hfill\\
{\bf Example 1.2}: We can do the same in any positive characteristic
$p$ to obtain the ``$p$-lamplighter group'' $C_p\,\wr\, \Z$, using the
matrices $\sma{cc}t&0\\0&1\fma$ and $\sma{cc}1&1\\0&1\fma$
over the field $\F_p(t)$. Again this is  a semidirect product
$\left(\bigoplus_{i=-\infty}^{i=+\infty}C_p\right)\rtimes\Z$ and has the
property that it is finitely generated and metabelian, but it
contains an abelian group which is torsion and not finitely
generated. Thus $C_p\,\wr\,\Z$ fails to be virtually torsion free.

We now consider distortion of cyclic and abelian subgroups. The next two
cases are by far the most commonly given examples of distorted infinite
cyclic subgroups.\\
\hfill\\
{\bf Example 1.3}: Consider the finitely presented Baumslag-Solitar
group\\ $BS(1,2)=\langle a,t\,|\,tat^{-1}=a^2\rangle$. This is also solvable
and metabelian and is even linear over $\Q$ using the matrices
$a=\sma{cc}1&1\\0&1\fma$ and  $t=\sma{cc}2&0\\0&1\fma$. 
Again it has an abelian subgroup which is not finitely generated. Moreover
as we clearly have $t^kat^{-k}=a^{2^k}$ for the infinite order element $a$,
the subgroup $\langle a\rangle$ is distorted in $BS(1,2)$. In fact the
behaviour is even worse because $BS(1,2)$ is an example of a group $G$
which is not balanced: namely there exists an element $a\in G$ of 
infinite order
such that $a^m$ is conjugate to $a^n$ but $|m|\neq |n|$ (whereupon
$\langle a\rangle$ is distorted in $G$ if $G$ is finitely generated). 
Such groups do not have good properties in general, for instance they
cannot be linear over $\Z$.\\
\hfill\\
{\bf Example 1.4}: Consider the Heisenberg group
$H=\langle x,y,z\,|\,z=[x,y],[z,x],$ $[z,y]\rangle$ 
(where $[x,y]=xyx^{-1}y^{-1}$) which is finitely
presented and nilpotent of class 2 (so every subgroup is also
finitely presented) but which is not virtually abelian. It is
linear over $\Z$ using the matrices
\[x=\sma{ccc}1&1&0\\0&1&0\\0&0&1\fma,
y=\sma{ccc}1&0&0\\0&1&1\\0&0&1\fma,
z=\sma{ccc}1&0&1\\0&1&0\\0&0&1\fma
\]
but as $z^{n^2}=x^{-n}y^{-n}x^ny^n$, we have that $\langle z\rangle$
is distorted in $H$. (Here $\langle z\rangle$ is said to
have polynomial distortion, whereas $\langle a \rangle$ in the previous
example has exponential distortion but in this paper we will not be concerned
with the type of distortion, rather just whether the subgroup is distorted
or not.)
Note also that if $S$ is any finite index subgroup of $H$ and we have
$l>0$ with $x^l,y^l\in S$ then $z^{l^2}$ is a commutator in $S$, so
no power of $z$ has infinite order in the abelianisation of $S$.\\
\hfill\\
{\bf Example 1.5}: Consider the semidirect product 
$S=\Z^2\rtimes_\alpha\Z=\langle a,b,t\rangle$
given by the automorphism $\alpha(a)=tat^{-1}=a^2b,\alpha(b)=
tbt^{-1}=ab$ of $\Z^2=\langle a,b\rangle$. This is polycyclic, 
so again every subgroup is finitely presented, but
not virtually nilpotent (and thus not virtually abelian). Moreover
$S$ is linear over $\Q(\phi)$, where $\phi$ is the golden ratio,
via the matrices
\[a=\sma{cc}1&1\\0&1\fma,b=\sma{cc}1&\phi-1\\0&1\fma,
t=\sma{cc}\phi+1&0\\0&1\fma.
\]

Now the abelian subgroup $\langle a,b\rangle\cong\Z^2$ is distorted
in $S$ because $t^nat^{-n}=a^{s_{2n+1}}b^{s_{2n}}$ for $s_n$ the $n$th
Fibonacci number. However every cyclic subgroup of $S$ is undistorted
(for instance by seeing $S$ as the fundamental group of a closed
Sol 3-manifold).

These examples, which were all in very small dimension and used basic fields, 
certainly demonstrate that small subgroups of finitely generated
linear groups need not have the strong behaviour that we see
elsewhere. However surely a striking
feature of all of these examples is that the ``trouble'' comes
from non identity matrices of the form
\[\sma{cc}1&?\\0&1\fma\mbox{ or }\sma{ccc}1&?&?\\0&1&?\\0&0&1\fma\]
which happen to be unipotent.

In this paper we will present two classes of 
linear groups, defined by restricting the unipotent matrices 
which are allowed to occur.
One of our classes contains the other but both have considerable
range in all dimensions and in any characteristic. Indeed the
more general class includes all linear groups in positive characteristic
and in zero characteristic it includes all real orthogonal
and complex unitary linear groups. The stricter class
is actually the same as the more general class in characteristic zero
whereas it certainly contains all linear groups in
positive characteristic which are virtually torsion free. We will show in 
Theorem \ref{undis} that if $G$ is a
finitely generated group in the more general class then every 
finitely generated virtually
abelian subgroup is undistorted in $G$, and in Corollary \ref{sbgp}
that a finitely generated
group in the stricter class has the ``small subgroups'' property
that every subgroup either contains $F_2$ or
is both finitely generated and virtually abelian
(thus also undistorted).

In fact the motivation for these results came not quite from
hyperbolic groups, but more from non positive curvature.
Our stricter class of linear groups, or at least the finitely
generated groups in this class, corresponds reasonably well to
the class of CAT(0) groups (those groups which act geometrically,
namely properly and cocompactly by isometries, on a CAT(0) space). 
Although neither class
in this correspondence contains the other, we will see that they
share some group theoretic properties. As for our wider class of linear
groups, in this paper they will be compared to 
groups $H$ having an isometric action on a CAT(0) metric
space $(X,d)$ where the action is both proper and, as a weakening
of being cocompact, is  semisimple
(meaning that for any $h\in H$ the displacement function $x\mapsto
d(x,h(x))$ attains its infimum over $X$). Again the two classes being
compared are not equal, for instance the
Burger-Mozes groups are CAT(0) and even have a geometric action on a
2 dimensional CAT(0) cube complex, but are infinite simple finitely
presented groups, so are as far from being linear as can be imagined.
However this time we will show that
the correspondence between the group theoretic
properties of these two classes is extremely close indeed. 
We now give more details of our results.   

First for CAT(0) groups, we have the following theorem
in \cite{bdhf} Part III Chapter $\Gamma$ Section 1:
\begin{thm} (\cite{bdhf} Theorem 1.1 Part 1) \label{know}
A CAT(0) group $\Gamma$ has the following properties:\\
(1) $\Gamma$ is finitely presented.\\
(2) $\Gamma$ has only finitely many conjugacy classes of finite
subgroups.\\
(3) Every solvable subgroup of $\Gamma$ is virtually abelian.\\
(4) Every abelian subgroup of $\Gamma$ is finitely generated.\\
(5) If $\Gamma$ is torsion-free, then it is the fundamental group of
a compact cell complex whose universal cover is contractible.
\end{thm}
Now none of these five points hold in general on replacing CAT(0) group
by finitely generated linear group, regardless of the characteristic.
For instance $F_2\times F_2$ is linear in any characteristic but
it has finitely generated subgroups which are not finitely presented,
thus (1) and (5) fail here on being applied to one of these
subgroups. Next Proposition 4.2 of \cite{grpl} displays
a finitely generated subgroup of $SL(2,\Z)\times SL(2,\Z)$ (also
linear in any characteristic) with infinitely many conjugacy classes of 
order 4
(extended in \cite{brima} to produce finitely presented subgroups 
of $SL(2,\Z)\times SL(2,\Z)\times SL(2,\Z)$ with the same property).
Moreover we have already seen the failure of (3) and (4) for the
wreath products in Examples 1.1 and 1.2. Indeed the counterexamples
being given here for Properties (1), (2) and (5) all belong in our
stricter class of linear groups. Nevertheless a consequence of Theorem
\ref{alt} in our paper is that (3) and (4) do actually hold for 
all finitely generated linear groups in our stricter class, so some of this
correspondence does go through.

However there is a Part 2 of the theorem in \cite{bdhf} as follows:
\begin{thm} (\cite{bdhf} Theorem 1.1 Part 2) \label{know2}
If $H$ is a finitely generated group that acts properly (but not necessarily
cocompactly) by semisimple isometries on the CAT(0) space $X$, then:\\
(i) Every polycyclic subgroup of $H$ is virtually abelian.\\
(ii) All finitely generated abelian subgroups of $H$ are
undistorted in $H$.\\
(iii) $H$ does not contain subgroups of the form
$\langle a,t|t^{-1}a^pt=a^q\rangle$ 
for non zero $p,q$ with $|p|\neq |q|$.\\
(iv) If $A\cong\Z^n$ is central in $H$ then there exists a subgroup
of finite index in $H$ that contains $A$ as a direct factor.\\
\end{thm}

Again all four of these points fail for finitely generated linear
groups overall, as already seen in the examples above.
But in this paper we will show that these four results hold verbatim
for any finitely generated group in our more general class of linear groups.
Consequently they all hold for any finitely
generated group which is linear in positive characteristic.
In fact we have already observed that Property (iii) is a consequence of 
Property (ii).
The rest of the paper is arranged as follows: in the next section
we introduce our two classes, where the more general class
will be known as NIU-linear groups (for no infinite order unipotents)
and the stricter class as VUF-linear groups (virtually unipotent
free). 
We show in this section that both classes of groups have good closure
properties, indeed the same as for
linear groups in general. Section 3 is about
centralisers where we establish Property (iv) for 
finitely generated NIU-linear groups
in Corollary \ref{coiv}.
  
Section 4 is about small subgroups. 
Property (i) is readily established for NIU-linear groups 
in Proposition \ref{coi} but we get much stronger behaviour
for VUF-linear groups, as Corollary \ref{sbgp} states that
a finitely generated VUF-linear group $G$ (in fact here finite
generation is required) has the property that
every subgroup $S$ of $G$ either
contains $F_2$ or is virtually abelian and finitely generated.
Then in Section 5 we show that for $G$ any finitely generated
NIU-linear group, all finitely generated abelian subgroups
of $G$ are undistorted in $G$. This extends a result in \cite{lmr}
which established this for cyclic subgroups. 

Of course both $G$ and the subgroup
must be finitely generated for this concept to be defined, but combining
this with Section 4 tells us that if $G$ is any finitely generated
VUF-linear group then any subgroup $S$ of $G$ is virtually abelian,
finitely generated and undistorted in $G$. Consequently we end up
showing that finitely generated VUF-linear groups are
extremely well behaved: not only do they have the
very strong properties on small subgroups enjoyed by word
hyperbolic groups, but they also have much better closure properties
than word hyperbolic groups.

We provide applications of each of these results in Section 6.
In particular, by using the ideas in \cite{bri} on centralisers
of Dehn twists in mapping class groups, we show in Corollary \ref{yoh}
that for genus at least 3 the mapping class group cannot be linear
over any field of positive characteristic, nor embed in the
complex unitary group of any finite dimension
(indeed any such representation in either case
sends all Dehn twists to elements
of finite order). We then consider how our classes of linear groups
correspond to the fundamental groups of 3-manifolds of non positive
curvature. Though we only consider closed 3-manifolds with a geometric
structure, we can show quickly in Theorem \ref{3m} that for such an
$M^3$ with $\pi_1(M^3)$ not virtually cyclic, the fundamental group
belongs in either of our two classes of linear groups 
if and only
if $M^3$ admits a Riemannian metric of non positive curvature.  

We then consider Euclidean groups, meaning arbitrary abstract subgroups
of the Euclidean isometry group in a particular dimension, without any
condition on their geometric properties. These need not be in either of
our two classes of linear groups, although by using our earlier results
we can reduce this to dealing with their translation subgroup. We establish
Property (iv) for any finitely generated Euclidean group in Corollary 
\ref{cene}  and a similar result on small subgroups in Corollary \ref{smae}
(though even in dimension two we can have finitely generated Euclidean
groups which are solvable but not virtually abelian). We also prove in
Corollary \ref{undse} that all infinite cyclic subgroups of finitely
generated Euclidean groups are undistorted, in contrast to \cite{conir}
which shows that finitely generated abelian subgroups are often distorted.

We finish by looking at the group $Out(F_n)$ which is not linear
(at least for $n\geq 4$). The paper \cite{alib} shows that all
infinite cyclic subgroups of $Out(F_n)$ are undistorted and this
was recently extended to abelian groups in \cite{wrig}, though in
Theorem \ref{trids} we show that this also follows by combining the original 
result in \cite{alib} and a quick argument using translation lengths.
We also note that the property of undistorted abelian subgroups then
extends immediately to $Aut(F_n)$ and free by cyclic groups $F_n\rtimes\Z$.

\section[Linear groups without infinite order unipotents]{Properties 
of linear groups without infinite order unipotents}

If $\F$ is any field and $d\in\N$ any positive integer
then we say an element $g$ of the general linear group $GL(d,\F)$ is 
{\bf unipotent}
if all its eigenvalues (considered over the algebraic closure
$\overline{\F}$ of $\F$) are equal to 1, or equivalently 
some positive power of $g-I$ is the zero matrix.

\begin{prop}  \label{diag} \qquad\\
(i) If $\F$ is a field of characteristic $p>0$  
and $M\in GL(d,\F)$ is unipotent then $M$ has finite order
equal to some power of $p$.

Conversely if $M$ is any element of $GL(d,\F)$ with order
$n$ which is a multiple of $p$ then $M^{n/p}$ is a 
non identity unipotent element. \\ 
\hfill\\
(ii) If $\F$ has zero characteristic then the only unipotent element
$M$ having finite order in $GL(d,\F)$ is $I$.
\end{prop}
\begin{proof}
For (i) there is $r>0$
with $N^l=0$ for $l\geq r$, where $N=M-I$. If
we take $k$ to be any power of $p$
which is at least $r$ then
\[M^k=(I+N)^k=I^k+\binom{k}{1}I^{k-1}N+\ldots
+\binom{k}{k-1}I N^{k-1}+N^k.\]
But $N^k=0$ because $k\geq r$ and $\binom{k}{i}\equiv 0$ modulo $p$
for $0<i<k$ as $k$ is a power of $p$, thus $M$ has order dividing $k$.

We now assume for the rest of the proof that $\F$ is algebraically
closed. As $1 \leq n/p<n$ we know that $M^{n/p}$ has order exactly $p$ and
hence has minimum polynomial $p(x)$ dividing $x^p-1=(x-1)^p$, but
any eigenvalue of a matrix must be a root of its minimum polynomial
so $M^{n/p}$ is unipotent.

For (ii), if $M$ has finite order then the minimum polynomial
of $M$ is $x^n-1$ for some $n\in\N$ and in characteristic zero this has 
no repeated roots, so $M$ is diagonalisable over $\F$ but has all eigenvalues
equal to 1, so is the identity. 
\end{proof}

We now come to the two key definitions of the paper.
\begin{defn} If $\F$ is any field and $d\in\N$ any dimension then we say
that a subgroup $G$ of $GL(d,\F)$ is {\bf NIU-linear} (standing for
linear with No Infinite order Unipotents) if every unipotent element of
$G$ has finite order.
\end{defn}
Note: by Proposition \ref{diag}, if $\F$ has positive characteristic then $G$ 
is automatically NIU-linear. If $\F$ has characteristic zero then the
definition says that the only unipotent element of $G$ is the identity.\\
\hfill\\
{\bf Example}: If $G$ is any subgroup of the real orthogonal group $O(d)$ or
of the complex unitary group $U(d)$ in any dimension $d$ then $G$ is
NIU-linear because all orthogonal or unitary matrices are diagonalisable
over $\C$.\\
\begin{defn} 
If $\F$ is any field and $d\in\N$ any dimension then we say
that a subgroup $G$ of $GL(d,\F)$ is {\bf VUF-linear} (standing for
linear and Virtually Unipotent Free) if $G$ has a finite index
subgroup $H$ where the only unipotent element of $H$ is the identity.
\end{defn}
Note: clearly VUF-linear implies NIU-linear and they are the same in
characteristic zero. As for the case when $G$ is linear in positive
characteristic, clearly if $G$ is also virtually torsion free then
it is VUF-linear. Although this need not true the other
way round, it does hold if $G$ is finitely generated, say by \cite{wehbk}
Corollary 4.8. This states that any
finitely generated linear group
has a finite index subgroup whose elements of finite order are
all unipotent (which might be thought of as ``Selberg's theorem in
arbitrary characteristic'').\\
\hfill\\
When we have a group $G$ which is only given in abstract form then to
say $G$ is NIU-linear or VUF-linear will mean that there exists 
some field $\F$ and dimension
$d$ such that $G$ has a faithful representation in $GL(d,\F)$ and the
image of this representation has the respective property.
We now begin by examining the closure properties of these two classes of
linear groups, which turn out to be the same as for
arbitrary linear groups. More precisely the property 
of being linear in a given
characteristic is known to be preserved under subgroups, commensurability
classes, direct and free products, so we now show that the same is true
for either of these two classes when restricted to a particular characteristic.

\begin{prop} \label{clos}
If $G_1$ and $G_2$ are groups which are both NIU-linear over fields
$\F_1,\F_2$ having the same characteristic then:\\
(i) Any subgroup $S$ of $G_1$ (or $G_2$) is NIU-linear.\\
(ii) Any group $G$ commensurable with $G_1$ (or $G_2$) is NIU-linear.\\
(iii) The direct product $G_1\times G_2$ is NIU-linear.\\
(iv) The free product $G_1*G_2$ is NIU-linear.

The same holds with NIU-linear replaced throughout by VUF-linear.
\end{prop}
\begin{proof} We will proceed in the following order. First, 
NIU-linear groups in a given positive characteristic
just mean arbitrary linear groups in this characteristic,
in which case the closure properties are already known.
We then argue for NIU-linear groups in characteristic
zero, which here are the same as VUF-linear groups. We then make the
necessary adjustments in our proof for VUF-linear groups in positive
characteristic.

Part (i) is immediate for NIU-linear groups and follows straight away
for VUF-linear groups because if $S$ is a subgroup of $G$ and $H$ is the
given finite index subgroup which is unipotent free then $S\cap H$
has finite index in $S$. This now reduces (ii) to saying that
if $G_1$ is NIU-linear and is also a 
finite index subgroup of the group $G$ then 
$G$ is NIU-linear too (and the same for VUF-linearity but this is immediate).
We certainly know that $G$ is linear over the same field as
$G_1$ by induced representations (though probably of bigger dimension).
But if $g\in G$ is a unipotent element of infinite order then so are all
its positive powers $g^n$ and some of these will lie in $G_1$.

For (iii) and (iv), we first observe that there is a field $\F$
containing both $\F_1$ and $\F_2$. (As they have the same prime subfield
$\PP$, we can adjoin enough transcendental elements to $\PP$ which are all
algebraically independent, resulting in a field $\F'$ where all elements of
$\F_1$ and $\F_2$ are algebraic over $\F'$, so that $\F_1$ and $\F_2$ both
embed in the algebraic closure of $\F'$.) We then see that the direct product
of $G_1\times G_2$ is linear over $\F$ in the usual way by combining
the two blocks representing $G_1$ and $G_2$. 
Then we
note that the eigenvalues of an element $g\in G_1\times G_2$ are just the
union of the eigenvalues in each of the two blocks. Thus a unipotent
element of $G_1\times G_2$ is unipotent in both the
$G_1$ and the $G_2$ blocks, hence in characteristic zero it
is the identity in $G_1\times G_2$ if
$G_1$ and $G_2$ are both NIU-linear. If instead they are VUF-linear with
finite index subgroups $H_1,H_2$ respectively that are unipotent
free then so is $H_1\times H_2$, which has finite index in $G_1\times G_2$.

Free products of linear groups over the same characteristic were shown to
be linear in \cite{nis}, which of course implies NIU-linearity
in positive characteristic. For the remaining cases
we can assume that we are in an algebraically closed field
$\F$ which has infinite transcendence degree
over its prime subfield. This is the setting for the proofs
in \cite{sha} which rediscovered the result on free products. (Actually
that paper works in $SL(d,\F)$ but the proofs go through in $GL(d,\F)$ as
well.)

We first assume, by increasing the size of the matrices and adding 1s
on the diagonal if needed, that the NIU-linear groups
$G_1$ and $G_2$ both embed in
$GL(d,\F)$ with neither subgroup containing any scalar matrices apart
from the identity. Moreover these embeddings will still be NIU-linear. 
Then Lemma 2.2 of the above paper says there is a
conjugate of $G_1$ (which will henceforth be called $G_1$) in $GL(d,\F)$
such that no non identity element of this conjugate has zero in its top
right hand entry. Similarly by taking a conjugate of $G_2$ we can assume
that no non identity element of $G_2$ has zero in its bottom left hand
entry. Next \cite{sha} Proposition 1.3 shows that $G_1*G_2$ embeds in
$GL(d,\F(t))$ for $t$ any element which is transcendental over $\F$. This
is achieved by conjugating $G_1$ by a diagonal matrix made up of powers of 
$t$ and then showing that in the resulting linear representation of 
$G_1*G_2$, any element $g\in G_1*G_2$ which is not conjugate into $G_1$
or $G_2$ will have a trace which is a Laurent polynomial in $t^{\pm 1}$
with at least two non trivial terms and so $g$ is not the identity. But
this also means that the trace of $g$ is transcendental over $\F$, thus
cannot equal $d$ and so $g$ is not unipotent. Now
being unipotent is a conjugacy invariant and
as there are no unipotents
in $G_1$ or $G_2$ either,
the resulting faithful linear representation of $G_1*G_2$ is
NIU-linear.

Thus we are done in characteristic zero for both our classes of linear
groups. As for preservation of VUF-linearity under free products in
characteristic $p>0$, we can use a trick: by dropping down further if
necessary we can assume that both the unipotent free finite index subgroups
$H_1$ of $G_1$ and $H_2$ of $G_2$ are normal. This then gives us two
homomorphisms (for $i=1,2$) $q_i:G_i\rightarrow G_i/H_i$ onto finite groups
and these can be both be extended from $G_1*G_2$ to $G_i/H_i$ with kernels
which we will call $K_1$ and $K_2$. Now note that $K_1\cap K_2$ is also
normal and has
finite index in $G_1*G_2$, and that both maps from $G_1*G_2$
to $G_i$ are retractions so we have $G_i\cap K_i=H_i$.

So suppose that there is a non identity unipotent element $k\in K_1\cap K_2$.
By Proposition \ref{diag} (i) we can assume that $k$ has order $p$. Thus
in the free product $G_1*G_2$ we must have that $k$ is conjugate into 
$G_1$ or $G_2$. This conjugate also lies in $K_1\cap K_2$ and is
unipotent, so if it is
in $G_1$ then it is also in $H_1$ and the same for $G_2$ and $H_2$. But
both $H_1$ and $H_2$ are unipotent free, so either way we are done.  
\end{proof}
\hfill\\
Note: To see that we cannot mix and match different characteristics
in (iii) and (iv), even for finitely generated groups, the lamplighter
group $C_2\,\wr\,\Z$ is linear in characteristic 2,
but only in this characteristic, whereas the ``trilamplighter'' group
$C_3\,\wr\,\Z$ is linear only in characteristic 3. Thus any group
containing them both (such as their direct or free product) is not
NIU-linear, or even linear over any field. 

We finish this section with a few words on how these closure properties
work for the three geometric classes of groups that were mentioned in the 
introduction. First, being word hyperbolic is well known to be preserved
under commensurability classes and free products, but hardly ever under
direct products. Moreover it is certainly not preserved under passage to
subgroups in general, but even if we only pass to finitely presented
subgroups then there are examples of Noel Brady which are not word
hyperbolic.

Next CAT(0) groups are preserved under both free and direct products, but 
again we can have finitely presented subgroups of CAT(0) groups which
are not themselves CAT(0), for instance by using finiteness properties
of direct products of free groups and their subgroups (see \cite{bdhf}
Chapter III.$\Gamma$ Section 5). Also commensurability is unclear
because of the problem of lifting the action of a finite index subgroup
to the whole group in a way that preserves cocompactness.

However we do obtain all four closure properties for the class of groups
acting properly and semisimply on a CAT(0) space, provided we impose the
mild extra condition that the CAT(0) space is complete
(indeed some authors refer to this as a Hadamard space). This geometric
property is preserved by free and direct products, and clearly also by
passage to arbitrary subgroups. But less obviously,
if a group $H$ has finite index $i$
in $G$ and $H$ has a proper action on the CAT(0) space $X$ then
we can use this to induce an action of $G$ on the direct product of $i$ copies
of $X$ (also a CAT(0) space) and this action will be proper. Moreover if
$X$ is complete then this induced action of $G$ will also be semisimple,
by \cite{bdhf} Chapter II Proposition 6.7 and Part (2) of Theorem 6.8
(we thank Martin Bridson for helpful correspondence on this point).

Consequently, as for the analogies we have mentioned
between the various classes of linear and
geometric groups, the strongest one by far is between our NIU-linear
groups and groups acting properly and semisimply on a complete CAT(0)
space. Having seen that the closure properties above all hold in both
of these cases, the next sections are about other group theoretic properties
of NIU-linear and VUF-linear groups. 

We finish this section by remarking
on two papers containing related results: first \cite{kaplb} considers
basically the same class of groups (which they call Hadamard groups)
as those acting properly and semisimply on complete CAT(0) spaces
(their definition of a proper action, which they call discrete, is that of 
acting
metrically properly, though if the space is itself a proper metric
space then the two definitions are equivalent). Also \cite{alsh} shows
that in characteristic zero a finitely generated linear group
has finite virtual cohomological dimension if and only if there is a
finite upper bound on the Hirsch ranks of its finitely generated
unipotent
subgroups. Thus our finitely generated NIU-linear groups in characteristic 
zero always have finite virtual cohomological dimension, though as this also
holds for the last three examples in Section 1, we see that this property on 
its own is not enough to rule out bad behaviour.

\section[Centralisers in linear groups]{Abelianisation of centralisers in 
linear groups with restricted unipotent elements}
We first show the following result. 
\begin{thm} \label{det}
Suppose that $G$ is a linear group over some field $\F$
of arbitrary
characteristic and $A$ is a abelian subgroup which is
central in $G$. Let $\pi$ be the homomorphism from $G$ to
its abelianisation $G/G'$. We have:\\
(i) If $G$ is NIU-linear then $\mbox{ker}(\pi)$ is a torsion group.\\
(ii) If $G$ is VUF-linear then $\mbox{ker}(\pi)$ is a finite group.
\end{thm}
\begin{proof}
We first replace our field by its algebraic closure, which we will also call
$\F$. Then it is true that any abelian subgroup of $GL(d,\F)$ is conjugate
in $GL(d,\F)$ to an upper triangular subgroup of $GL(d,\F)$, for instance
by induction on the dimension and Schur's Lemma.

For any $g\in G$ and $a\in A$ we have $ga=ag$. This means that
$g$ must map not just each eigenspace of $a$ to itself, but each
generalised eigenspace
\[E_{\lambda}(a)=\{v\in\F^d:(a-\lambda I)^nv=0\mbox{ for some } n\in\N\}
\mbox{ where }\lambda\in\F\]
and together these span, so that if $a$ has distinct eigenvalues
$\lambda^{(a)}_1,\ldots ,\lambda^{(a)}_{d_a}$ then 
$\bigoplus_{i=1}^{d_a} E_{\lambda^{(a)}_i}(a)$ is a
$G$-invariant direct sum of $\F^d$.

We now take a particular (but arbitrary) non identity element $a$ of $A$
and restrict $G$ to the first of these
generalised eigenspaces $E_{\lambda^{(a)}_1}(a)$, so that 
here $a$ only has the one eigenvalue $\lambda^{(a)}_1$.  
If this property also holds on $E_{\lambda^{(a)}_1}(a)$ for every
other $a'\in A$ then we proceed to $E_{\lambda^{(a)}_2}(a)$, 
$E_{\lambda^{(a)}_3}(a)$ and so on. Otherwise there is another $a'\in A$
such that we can split $E_{\lambda^{(a)}_1}(a)$ further into
pieces where $a'$ has only one eigenvalue on each piece.
Moreover this decomposition is also $G$-invariant because it can
be thought of as the direct sum of the generalised eigenspaces of $a'$
when $G$ is restricted to $E_{\lambda^{(a)}_1}(a)$. 

We then continue this process on all of the pieces and over all elements
of $A$ until it terminates (essentially we can view it as building a
rooted tree where every vertex has valency at most $d$ and of finite
diameter). We will now find that we have split $\F^d$ into a $G$-invariant
sum $V_1\oplus\ldots \oplus V_k$ of $k$ blocks, where any element
of $A$ has a single eigenvalue when restricted to any one of these 
blocks. 

Now we conjugate within each of these blocks so that the restriction
of $A$ to this block is upper triangular, using the comment at the
start of this proof. Under this basis so obtained
for $\F^d$, we have that any $a\in A$ will now be of the form
\[a=\sma{ccc}\boxed{T_1}& &0\\
&\ddots&\\
0& &\boxed{T_k}\fma\]
where each block $T_i$ is an upper triangular matrix with all diagonal entries
equal (as these are the eigenvalues of $a$ within this block).
More generally any $g\in G$ will be of the form
\[\sma{ccc}\boxed{M_1}& &0\\
&\ddots&\\
0& &\boxed{M_k}\fma
\]
for various matrices $M_1,\ldots ,M_k$ which are the same size as the
respective matrices $T_1,\ldots ,T_k$ because we know $g$ 
preserves this decomposition.

Consequently we have available as homomorphisms from $G$ to the multiplicative
abelian group $(\F^*,\times)$ not just the determinant itself but also
the ``subdeterminant'' functions
$\mbox{det}_1,\ldots ,\mbox{det}_k$, where for $g\in G$ the
function $\mbox{det}_j(g)$ is defined as
the determinant of the $j$th block of $g$ when
expressed with respect to our basis above, and
these are indeed homomorphisms as is
\[\theta:G\rightarrow (\F^*)^k\mbox{ given by }
\theta(g)=(\mbox{det}_1(g),\ldots ,\mbox{det}_k(g)).\]
As $\theta$ is a homomorphism from $G$ to an abelian
group, it factors through the homomorphism $\pi$ 
from $G$ to its abelianisation because this is the universal abelian
quotient of $G$. This means that $\mbox{ker}(\pi)$ is contained in 
$\mbox{ker}(\theta)$ and so we can replace $\pi$ with $\theta$ for the 
rest of the proof.

Thus suppose that there is some $a\in A$ which is in the kernel of 
$\theta$. We know that
\[a=\sma{ccc}\boxed{T_1}& &0\\
&\ddots&\\
0& &\boxed{T_k}\fma\]
for upper triangular matrices $T_i$
and as the diagonal entries of $T_i$ are constant, say $\mu_i$
for $\mu_1,\ldots ,\mu_k\in \F^*$,
we conclude that $\mu_i^{d_i}=1$ where $d_i=\mbox{dim}(V_i)$.
In other words a power of $a$ is
unipotent, thus if $G$ is NIU-linear then $a$ has finite order. 

In the case where $G$ is VUF-linear with a finite index subgroup
$H$ having no non trivial unipotent elements, we have that $H\cap A$ has
finite index in $A$ so we will
take the restriction of $\theta$ to $H\cap A$
and show that this has finite kernel. 
If we have elements $u_1,u_2\in H\cap A$ with exactly
the same diagonal entries then $u_1^{-1}u_2$ must be unipotent and so
$u_1=u_2$. But on considering the diagonal entries
of an upper triangular 
element in $\mbox{ker}(\theta)$, we see they are all roots of unity
with bounded exponent
and so there are only finitely many possibilities, thus also only
finitely many possibilities for elements of $H\cap A$ which are also
in the kernel of $\theta$.
\end{proof}

In particular any infinite order element of $A$ also has infinite order
in the abelianisation of $G$. We now adapt this to obtain the same
conclusion of Theorem \ref{know2} Part (iv) in the case of NIU-linear groups. 
\begin{co} \label{coiv}
If $H$ is a finitely generated NIU-linear group and 
$A\cong\Z^n$ is central in $H$ then there exists a subgroup
of finite index in $H$ that contains $A$ as a direct factor.
\end{co}
\begin{proof} Theorem \ref{det} gives us a homomorphism $\theta$ from 
$H$ to some abelian group $C$ which is injective on $A$. By dropping to
the image, we can assume that $\theta$ is onto without loss
of generality and so $C$ is also finitely generated. By the classification
of finitely generated abelian groups, we have that $C=\Z^m\oplus$Torsion
for $m\geq n$ and we can compose $\theta$ with a homomorphism $\phi$ from
$C$ to $\Z^n$ in which $A$ still injects, so will have finite index.

Thus if we set $K=\mbox{Ker}(\phi\theta)$ then the pullback 
$(\phi\theta)^{-1}(A)=KA$ has finite index in $H$. Also $K$ and $A$ are
normal subgroups of $H$ with $K\cap A=\{e\}$, giving $KA\cong K\times A$.
\end{proof}

\section[Small subgroups of linear groups]{Small subgroups of 
NIU-linear and VUF-linear groups}

Suppose that $\F$ is an algebraically
closed field of any characteristic, and that $S$ is a solvable subgroup
of $GL(d,\F)$.

A consequence of the Lie - Kolchin Theorem, or alternatively results 
of Malce'ev, is that there is a finite index subgroup $T$
of $S$ which is upper triangularisable, namely conjugate in $GL(d,\F)$ to
a subgroup of $GL(d,\F)$ where every element is upper triangular.
Thus on assuming that we have conjugated $S$ within $GL(d,\F)$
so that $T$ is in this upper triangular form, we immediately see 
there is a homomorphism $h$ from $T$ to the abelian
group $(\F^*)^d$ given by the diagonal elements of an element $t\in T$. 
As the kernel of $h$ consists only of upper unitriangular matrices, meaning
that it must be a nilpotent group, we obtain the well known:
\begin{prop} \label{solv}
A solvable linear group $S$ over an arbitrary field is
virtually (nilpotent by abelian), meaning that $S$ possesses a
finite index subgroup which has an abelian quotient with
nilpotent kernel.
\end{prop}

To improve on this result we will first use NIU-linearity and 
VUF-linearity, then further
strengthen it by assuming that $S$, which might not be finitely
generated, is in fact a subgroup of a finitely generated linear
group.
\begin{co} \label{cos}
Suppose that $S$ is a solvable group and is NIU-linear.
Then $S$ is virtually (torsion by abelian) and if we are in characteristic
zero then $S$ is in fact virtually abelian (although $S$ need not be 
finitely generated in any characteristic).
\end{co}
\begin{proof} The kernel of $h$ in Proposition \ref{solv} consists
entirely of unipotent elements. But as $S$ is
NIU-linear, we have by Proposition \ref{diag} that $\mbox{ker}(h)$ is a
torsion group in positive characteristic and $\mbox{ker}(h)=\{id\}$
in zero characteristic.

We have already seen that the wreath product $C_p\,\wr\,\Z$ is
linear in characteristic $p>0$ and this contains the infinitely
generated abelian group $\bigoplus_{i=-\infty}^{+\infty} C_p$.
As for characteristic zero, 
we can take all diagonal matrices over $\F=\Q$ say in any particular
dimension to get an NIU-linear and VUF-linear
group which is countable and abelian but not finitely generated.
\end{proof}

Next we show the equivalent result for NIU-linear groups of Theorem
\ref{know2} Part (i).
\begin{prop} \label{coi}
Suppose that $S$ is polycyclic and NIU-linear, then $S$ is virtually abelian.
\end{prop}
\begin{proof}
We are done in characteristic zero by Corollary \ref{cos} above.
Moreover $S$ being polycyclic means that all subgroups of $S$ are 
finitely generated, in particular $\mbox{ker}(h)$ which is also
a solvable torsion group and thus is finite.
Hence $T$ is finite by abelian as well as finitely generated, so by
standard results it is virtually abelian (for instance: $T$ is
certainly residually finite, so we can take finitely many
finite index subgroups of $T$, each missing an element of $\mbox{ker}(h)$,
and their intersection injects under $h$ so is abelian).
\end{proof}

We now move to considering solvable groups $S$ which are VUF-linear.
We can obtain strong results if $S$ is
finitely generated, but in fact it is enough to assume that $S$
is merely contained in some linear group $G$ which is finitely generated.   
This is because we can then
utilise the fact that $G$
can be thought of as a subgroup not only of $GL(d,\F)$ for $\F$ a
finitely generated field, but (by taking the ring generated by all
entries of a generating set for the group which is closed under inverses)
also of $GL(d,R)$ where $R$ is an integral domain which is finitely
generated as a subring of $\F$. This approach is exploited in \cite{wehbk}
Chapter 4 and allows us here to obtain a much better result on small
subgroups in line with word hyperbolic or CAT(0) groups.
\begin{thm} \label{alt}
Suppose that the solvable group $S\leq GL(d,\F)$ is VUF-linear. Then
$S$ is virtually abelian. If further we have that $S$ is a subgroup of
$G\leq GL(d,\F)$ where $G$ is finitely generated 
then $S$ is also finitely generated.
\end{thm}
\begin{proof}
We first replace $S$ with the appropriate finite index subgroup
$P$ which is unipotent free. 
We next assume that $\F$ is algebraically closed and proceed as in
Proposition \ref{solv} and Corollary \ref{cos} by  
conjugating $S$ in $GL(d,\F)$ so that $P$ has a finite
index subgroup $T$ which is upper triangular. Now
the homomorphism $h$ from $T$ to $(\F^*)^d$ has 
kernel consisting only of unipotent elements but $T$ is unipotent free,
so $T$ is abelian with $P$ and $S$ being virtually abelian.

Now suppose that $S$ is contained in the subgroup $G$ of $GL(d,\F)$
with $G$ finitely generated. We then follow
\cite{wehbk} Lemma 4.10 and see that, as $T$ is a subgroup of
$GL(d,R)$ for $R$ a finitely generated subring of the field $\F$
obtained from the entries of a symmetric generating set for $G$,
our homomorphism $h$ actually has image
in ${\cal U}^d$ for $\cal U$ the group of units of $R$, which
happens to be a finitely generated abelian group,
so $T$ and thus $P$ and $S$ are also finitely generated.
\end{proof}

We can now make this result on small subgroups of VUF-linear groups
definitive by bringing in the Tits alternative.

\begin{co} \label{sbgp}
Suppose that $G\leq GL(d,\F)$ is a finitely generated VUF-linear group.
Then any subgroup $S$ of $G$ (whether finitely generated
or not) satisfies the following alternative: either $S$
contains a non abelian free subgroup or $S$ is virtually abelian
and finitely generated.  
\end{co}
\begin{proof} The Tits alternative in characteristic zero tells
us that any linear group in this characteristic either contains a
non abelian free subgroup or is virtually solvable, thus 
Theorem \ref{alt} applies.

This does not quite work in positive characteristic, because although
the Tits alternative still holds as above for linear groups in positive
characteristic which are finitely generated, for an arbitrary
linear group we might only conclude that it is solvable by locally
finite (for instance $GL(d,\overline{\F_q})$ for $d\geq 2$ and
$\overline{\F_q}$
the algebraic closure of the finite field $\F_q$). However we
can again evoke finite generation of the ring $R$ in this situation
to conclude that if $S$ is a subgroup of a finitely generated group
which is linear in positive characteristic then either $S$ contains
a non abelian free subgroup or $S$ is indeed virtually solvable.
For instance \cite{wehbk} Lemma 10.12 states that if $R$ is a
finitely generated integral domain and $S$ is a subgroup of
$GL(d,R)$, which is the case for $S$ here, such that every 2-generator
subgroup of $S$ is virtually solvable, which is also the case for $S$
here if it does not contain a non abelian free subgroup, then 
$S$ itself is virtually solvable. 
\end{proof}

We note that in other settings we do not always have a Tits alternative
available. For instance it is currently unknown whether every finitely
generated subgroup of a CAT(0) group must either contain $F_2$ or
be virtually solvable, whereas there are finitely generated groups
acting properly and semisimply on CAT(0) spaces which fail this
alternative.

\section{Undistorted abelian subgroups}

In this section we prove Theorem \ref{undis} which states that
for a finitely generated NIU-linear group $G$,
all finitely generated abelian subgroups are
undistorted in $G$. Proposition 2.4 of \cite{lmr} 
showed this for cyclic subgroups
(in fact their statement concludes that cyclic subgroups
do not have exponential distortion but the proof given works for arbitrary
distortion). They use a fact from the paper \cite{tts}
establishing the Tits alternative which is that for any element of
infinite multiplicative order in a finitely generated field,
we can find an absolute value on the field which is not equal to
1 on this element. We will argue in this way but will need to work with
many absolute values simultaneously, as well as needing a separate
argument for units in number fields.

Suppose we have a group $G$ which is finitely generated but otherwise
arbitrary for now, along with a finitely generated abelian subgroup
$A$ of $G$. We will repeatedly use the fact that showing a finite index
subgroup of $A$ is undistorted in $G$ also establishes this for $A$ itself,
so without loss of generality we will say that 
$A=\langle a_1,\ldots ,a_m\rangle\cong\Z^m$ for some $m\geq 1$.

Suppose further that we have a function $f:G\rightarrow [0,\infty)$
with the following two properties:
\begin{eqnarray*}
&(i)&\mbox{ Subadditivity: }
f(gh)\leq f(g)+f(h)\mbox{ for all }g,h\in G\\
&(ii)&
\mbox{ Lower bound on $A$: 
There exists a real number $c>0$ such that for}\\
&&\mbox{ all }
(n_1,\ldots ,n_m)\in\Z^m,\mbox{ we have }
f(a_1^{n_1}\ldots a_m^{n_m})\geq c(|n_1|+\ldots+|n_m|).
\end{eqnarray*}
Then on taking any finite generating set $S$ for $G$ and denoting
$l_S(g)$ for the standard word length of an element $g\in G$ with respect
to $S$, we have by (i) that $l_S(g)C\geq f(g)$ on setting
$C=\mbox{max}\{f(s^{\pm 1}):s\in S\}$. Now $C\geq 0$ and indeed
$C>0$ unless $f\equiv 0$, in which case (ii)
cannot hold. Thus if
$a=a_1^{n_m}\ldots a_m^{n_m}$ is an arbitrary element of $A$ then 
\[l_S(a)\geq \frac{1}{C}f(a)\geq\frac{c}{C}(|n_1|+\ldots +|n_m|)
\mbox{ with }c/C>0,\]
so $A$ is undistorted in $G$.

For linear groups, at least over $\C$, we have a natural candidate
for this function $f$. Indeed if we now suppose that $G$ is a subgroup of
$GL(d,\C)$ then we can put the max norm $||\cdot ||_\infty$ on $\C^d$
and the corresponding operator norm
\[||M||_{\mbox{op}}=\mbox{sup} \{||Mv||_\infty/||v||_\infty:v\in \C^d,
v\neq 0\}\]
on $GL(d,\C)$ which satisfies 
$||M_1M_2||_{\mbox{op}}\leq ||M_1||_{\mbox{op}}||M_2||_{\mbox{op}}$
for $M_1,M_2\in GL(d,\C)$ (or even for two $d$ by $d$ matrices over
$\C$), so the operator norm is a map from $G$ to $(0,\infty)$ which
is submultiplicative. Of course this can be made subadditive by taking
logs, but it might return negative values. Therefore we actually define
\[f(M)=\mbox{log max}(||M||_{\mbox{op}},||M^{-1}||_{\mbox{op}}).\]
This will be at least zero because if $\lambda$ is
an eigenvalue with eigenvector $v\neq 0$, so that $Mv=\lambda v$,
then $||M||_{\mbox{op}}\geq |\lambda|$, but if $0<|\lambda|\leq 1$ then
$1/\lambda$ is an eigenvalue of $M^{-1}$ with $|1/\lambda|\geq 1$.

Now subadditivity for $f$ is easily checked using the submultiplicity
of $||\cdot||_{\mbox{op}}$, so we first review the proof
of Proposition 2.4 in \cite{lmr} for the case when our abelian
group $A=\langle M\rangle$ is infinite cyclic: if
$M\in GL(d,\C)$ possesses
an eigenvalue $\lambda$ with $|\lambda|>1$
then for $n>0$ we get $||M^n||_{\mbox{op}}\geq |\lambda|^n$, and if
$|\lambda|<1$ then $||M^{-n}||_{\mbox{op}}\geq |1/\lambda|^n$
so that 
\[f(M^n)=\mbox{log max}(||M^n||_{\mbox{op}},||M^{-n}||_{\mbox{op}})
\geq n\big|\mbox{log}|\lambda|\big|\]
and we can swap $M$ and $M^{-1}$ when $n<0$, hence $f(M^n)\geq c|n|$ for
$c=\big|\mbox{log}|\lambda|\big|$. Thus when $|\lambda|\neq 1$ we
will have $c>0$ and so $\langle M\rangle$ is undistorted in $G$.

However it is quite possible that all eigenvalues of
every element of $G$ have modulus 1, for instance if $G$ were
a group of orthogonal or unitary matrices (which is exactly
one of the cases we are considering).
We also need a generalisation of this argument to arbitrary fields, not just
the characteristic zero case.
In order to complete the proof for cyclic 
subgroups, \cite{lmr} refers back to the
famous proof \cite{tts} of the Tits alternative, and specifically 
Lemma 4.1 of this paper
which states that for any finitely generated field $\F$ and for any
element $z\in\F^*$ of infinite multiplicative order, there is an absolute
value $\ep$ on $\F$ with $\ep(z)\neq 1$. Here  if $\F$
is any field then an absolute value $\ep:\F\rightarrow [0,\infty)$
satisfies
\begin{eqnarray*}
&(i)&\ep(x)=0\mbox{ if and only if }x=0\\
&(ii)&\ep(xy)=\ep(x)\ep(y)\\
&(iii)&\ep(x+y)\leq \ep(x+y).
\end{eqnarray*} 

Thus to finish the cyclic case, suppose that
our group $G$ is a subgroup of $GL(d,\F)$ for $\F$
an arbitrary  field. We first note that
we can take $\F$ to be a finitely generated field without
loss of generality because $G$ is a finitely generated group.
Then for any element $M\in G$ we extend $\F$ to
$\F_M$ by adjoining the eigenvalues of $M$, so that $\F_M$ is still
finitely generated. Now if all eigenvalues of $M$ are roots of
unity then a power of $M$ is unipotent, which we are specifically
excluding. Thus there must be an eigenvalue $\lambda$ of $M$ and
an absolute  value $\ep$ on $\F_M$ such that $\ep(\lambda)\neq 1$, so 
we do indeed have that $|\mbox{log}\,\ep(\lambda)|=c>0$ in the above
and $\langle M\rangle$ is undistorted in $G$.

Now we move to the case when $A$ is a free abelian subgroup of rank $m$ in the
finitely generated group $G\leq GL(d,\F)$. However, although still following
the same path, the argument requires more work than the cyclic case.
In particular, difficulties are caused by the fact that
we can have $m>d$ in general as the next example shows:\\
\hfill\\
{\bf Example}: Let $\F=\Q$ and consider $\Z^2\cong 
A=\langle 2,3\rangle=G\leq
\Q^*=GL(1,\Q)$. Then on trying the obvious absolute value
$\ep(q)=|q|$ for $q\in \Q$ and setting
\[f(q)=\mbox{log }\mbox{max }(||q||_{\mbox{op}},||q^{-1}||_{\mbox{op}})
=\big|\mbox{log}|q|\big|,\]
Property (ii) above would require $c>0$
such that $f(2^{n_1}3^{n_2})\geq c(|n_1|+|n_2|)$. But 
$f(2^{n_1}3^{n_2})=|n_1\mbox{log}(2)+n_2\mbox{log}(3)|$ so this is
impossible because  
$\langle\mbox{log}(2),\mbox{log}(3)\rangle$ is 
dense in $\R$. However we also have
$p$-adic evaluations on $\Q$, in particular the 2-adic evaluation
$|\cdot|_2$ and the 3-adic evaluation $|\cdot|_3$. Let us now 
define $f$ as above but using either of these two evaluations in place of the
usual Euclidean absolute value. Then $f$ is still subadditive
but on setting $g=2^{n_1}3^{n_2}\in\Q^*$, we have
\[f(g)=\big|\mbox{log}(|g|_2)\big|=|n_1|\mbox{log}(2)
\mbox{ or }f(g)=\big|\mbox{log}(|g|_3)\big|=|n_2|\mbox{log}(3).
\] 
However in the first case we have $f(3^{n_2})=0$ and $f(2^{n_1})=0$
in the second, so Property (ii) fails badly here. Moreover
by Ostrowski's Theorem, every absolute value on $\Q$ is equivalent
either to the modulus or to a $p$-adic evaluation for some prime $p$,
all of which give rise to our function $f$ being subadditive but 
all fail Property (ii). The key now is to see that we can combine
different functions $f$ obtained from separate evaluations, because
the sum of two subadditive functions is also subadditive.

In particular, here we can set $f$ to be the function 
\[f(g)=\big|\mbox{log}(|g|_2)\big|+\big|\mbox{log}(|g|_3)\big|=
|n_1| \mbox{log}(2)
+|n_2| \mbox{log}(3)\]
which (as $\mbox{log}(2)<\mbox{log}(3)$) is
at least $(|n_1|+|n_2|)\mbox{log}(2)$,
thus providing the required lower bound on $A$. Of course as $A=G$
we have merely shown that $A$ is undistorted in itself, but we have used
Properties (i) and (ii) to establish this.

We can now proceed with the general argument. We first introduce
the functions $f$ that we will be using here.
\begin{prop} \label{subadd}
Let $\F$ any field, with $\ep_1,\ldots ,\ep_N$
any finite list of absolute values on $\F$.

Then for any dimension $d$
there exists a function $f:GL(d,\F)\rightarrow [0,\infty)$ which is
subadditive and with the following property: if $M\in GL(d,\F)$ has an
eigenvalue $\lambda\in\F^*$ then for all $1\leq i\leq N$ we get
$f(M)\geq\big|\log\ep_i(\lambda)\big|$.
\end{prop}
\begin{proof}
We first let $||\cdot ||_i$ be the max norm on the  vector space $\F^d$
with respect to $\ep_i$, so that
given $v=(v_1,\ldots ,v_d)\in \F^d$ we have
\[||v||_i=\mbox{max}\{\ep_i(v_1),\ldots ,\ep_i(v_d)\}.\]
We then let $||\cdot ||_{\mbox{op},i}$ be the corresponding
operator norm on $GL(d,\F)$, namely
\[||M||_{\mbox{op},i}=\mbox{sup} \{||Mv||_i/||v||_i:v\in \F^d,
v\neq 0\}\]
and define $f_i:GL(d,\F)\rightarrow [0,\infty)$ like before as
\[f_i(M)=\mbox{log }\mbox{max}(||M||_{\mbox{op},i},
||M^{-1}||_{\mbox{op},i})\]
for $M\in GL(d,\F)$. Then each $f_i$ is subadditive, thus so is
$f:=f_1+\ldots +f_N$. Now if $\lambda$ is an eigenvalue for $M$
then for each $i$ we certainly have
\[||M||_{\mbox{op},i}\geq\ep_i(\lambda)\mbox{ and }
||M^{-1}||_{\mbox{op},i}\geq 1/\ep_i(\lambda)\]
thus $f(g)\geq f_i(g)\geq \big|\mbox{log}\,\ep_i(\lambda)\big|$.
\end{proof}
We now come to the proof of the equivalent of Theorem \ref{know2} Part (ii).
\begin{thm} \label{undis} 
If $G$ is any finitely generated group which is NIU-linear
and $A$ is any finitely generated abelian subgroup
of $G$ then $A$ is undistorted in $G$.
\end{thm}
\begin{proof}
We take $G$ to be a subgroup of $GL(d,\F)$, where $\F$ is some field
which we initially assume to be algebraically closed.
By dropping to a finite index subgroup of $A$ if required, we can assume
that $A=\langle a_1,\ldots ,a_m\rangle$ is free abelian of rank $m$.
Moreover
we can conjugate $G$ within $GL(d,\F)$ so that every element
of $A$ is upper triangular. Having done this, we henceforth assume that
$\F$ is the finitely generated field over the relevant prime subfield
which is generated by the entries of $G$.

By taking logs, any absolute value $\ep_i$ on $\F$ can also be thought of 
as a group homomorphism $\mbox{log}\,\ep_i$ from $\F^*$ to $\R$. Moreover
from a single absolute value $\ep_i$, we obtain $d$ functions $\theta_{i,j}$
from $A$ to $\R$ where $1\leq j\leq d$ by setting 
$\theta_{i,j}(a)=\mbox{log}\,\ep_i(a_{jj})$ for $a_{jj}\in \F^*$ the $j$th
diagonal entry of the element $a\in A$. As any $a$ is an upper triangular 
matrix, we see that $\theta_{i,j}:A\rightarrow\R$ is actually a
homomorphism.
\begin{lem} \label{inj}
There exist absolute values $\ep_1,\ldots ,\ep_m$ on $\F$
such that if we let $\theta:A\rightarrow\R^{md}$ be the homomorphism
obtained from the coordinate functions $\theta_{i,j}$ then $\theta$ is
injective.
\end{lem}
\begin{proof} We proceed by induction on $m$. We first write
$A=A_0\oplus\langle a_m\rangle$ for $A_0=\langle a_1,\ldots,a_{m-1}\rangle$
and we can
assume that there exist absolute values $\ep_1,\ldots ,\ep_{m-1}$
such that the resulting homomorphism $\theta_0:A\rightarrow \R^{(m-1)d}$,
obtained by combining the relevant
coordinate functions $\theta_{i,j}$, is injective
on $A_0$. If it is injective on $A$ too then we are happy and can just
add any $\ep_m$ to the list of absolute values. Otherwise $\theta_0$
vanishes on some element $a$ of $A$ which must be of the form 
$a=a_0a_m^n$ for
$a_0\in A_0$ and $n\neq 0$. We now pretend that actually $A$ was the
finite index subgroup $\langle a_1,\ldots ,a_{m-1},a_m^n\rangle$ which is
also equal to $\langle a_1,\ldots ,a_{m-1},a\rangle$
(with the subgroup $A_0$ unchanged).

Next we use Tit's Lemma 4.1 in \cite{tts}: as $\F$ is assumed
finitely generated, we have that for any $z\in\F^*$ of infinite order,
there exists an absolute value $\omega$ on a locally compact field containing
$\F$, and thus on $\F$ itself by restriction, such that $\omega(z)\neq 1$.
We now apply this to each diagonal element of $a$: if every one has finite
order then the eigenvalues of $a$ are all roots of unity and so some 
power of $a$ is unipotent, which is excluded by the NIU-linear
hypothesis unless $a$ has finite order but $A$ is free abelian of
rank $m$.

Consequently there is some $1\leq k\leq d$ where 
the $k$th diagonal element $a_{kk}\in\F^*$ of $a$ has infinite
order, thus there is also an absolute value 
$\omega:\F\rightarrow[0,\infty)$ with
$\omega(a_{kk})\neq 1$. We set $\ep_m=\omega$ so that 
$\theta_{m,k}(a)=\mbox{log}\,\ep_m(a_{kk})\neq 0$. Hence on letting
$\theta$ be the extension of $\theta_0$ from $A$ to $\R^{md}$ obtained
by including $\theta_{m,1},\ldots ,\theta_{m,d}$ as extra coordinate
functions, we have that $\theta$ is injective on $A$: suppose $\theta$
vanishes on an element of $A$, which can also be written as $a_0a^n$ for
some $a_0\in A_0$, so that
$\theta(a_0a^n)=0\in \R^{md}$. Then
looking at the first $(m-1)d$ coordinate functions $\theta_{i,j}$ tells
us that $a_0$ is the identity (because we suppose that here $a$ vanishes
under $\theta_0$ but that $A_0$ injects), whereas 
$\theta_{m,k}(a^n)=0$ implies $n=0$.
\end{proof}

An absolute value $\ep$ on a field $\F$ is called {\bf discrete}
if the image of $\ep:\F^*\rightarrow (0,\infty)$ is a discrete set,
which is equivalent to saying that there exists $\lambda>0$ such that
the image of the function $\mbox{log}\,\ep$ in $\R$ is $\lambda\Z$.
We now proceed on the assumption that all absolute values obtained
in Lemma \ref{inj} are discrete (thus in $\Q$ this would mean we
have taken the $p$-adic evaluations but not the modulus) and we
will remove this assumption at the end of the proof. 

We have the coordinate functions $\theta_{i,j}$ on $A$ and we can
now assume that there exists real numbers $\lambda_{i,j}>0$
with $\theta_{i,j}(A)\leq\lambda_{i,j}\Z\leq\R$. Thus by Lemma \ref{inj}
$A$ embeds via $\theta$ in the subgroup
\[\prod_{1\leq i\leq m,1\leq j\leq d}
\lambda_{i,j}\Z\qquad\mbox{ of }\R^{md}\]
which is clearly a lattice in $\R^{md}$, namely  a discrete subgroup
of $\R^{md}$ which spans. Thus on letting $V$ be the vector subspace
of $\R^{md}$ spanned by $\theta(A)$, which is itself a lattice
in $V$, we see that for $a=a_1^{n_1}\ldots a_m^{n_m}\in A$ and for
any norm $||\cdot ||$ on $V$, we have $c>0$ such that
\[||\theta(a)||\geq c(|n_1|+\ldots +|n_m|).\]
Thus if we put the norm on $V$ induced from the max norm on $\R^{md}$,
we obtain
\[\mbox{max}\,\{|\theta_{i,j}(a)|:1\leq i\leq m,1\leq j\leq d\}
\geq c(|n_1|+\ldots +|n_m|).\]
But $\theta_{i,j}(a)=\log\,\ep_i(a_{jj})$ for $a_{jj}$ the $j$th diagonal
entry of $a$. As $a$ is upper triangular, this is an eigenvalue of
$a$. Hence by applying Proposition \ref{subadd} with
$\ep_1,\ldots ,\ep_m$, we obtain our subadditive function
$f:G\rightarrow [0,\infty)$ where for any $a=a_1^{n_1}\ldots a_m^{n_m}\in A$
and $1\leq i\leq m,1\leq j\leq d$, we have
$f(a)\geq \big|\log\,\ep_i(a_{jj})\big|$ so that
\[f(a)\geq c(|n_1|+\ldots +|n_m|)\mbox{ for some }c>0.\]
Thus $f$ satisfies both Property (i) and Property (ii) and hence $A$ is
undistorted in $G$.
 
We must now consider when we are able to use absolute values which
are discrete, thus we turn again to \cite{tts} Lemma 4.1 but this
time we examine the proof rather than just the statement. Our field 
$\F$ is finitely generated over its prime subfield
$\PP$ (namely $\PP=\Q$ in characteristic zero and $\PP=\F_p$ in
characteristic $p>0$) and we suppose we are given a non zero element $t\in \F$ 
which is not a root of unity (so $t$ has infinite order in $\F^*$).
Let us first assume that $t$ is transcendental over $\PP$, in which case
this proof proceeds as follows. We set $\F_{\mbox{alg}}$ to be the subfield of
$\F$ consisting of those elements which are algebraic over $\PP$, with
$\F_{\mbox{alg}}$ being finitely generated over $\PP$ because $\F$ is 
and so $\F_{\mbox{alg}}$ is a finite extension of $\PP$. Moreover 
finite generation of $\F$ (now over $\F_{\mbox{alg}}$) 
implies that we have a finite
transcendence basis of $\F$ over $\F_{\mbox{alg}}$, namely
a finite number of elements $t_1,\ldots ,t_u\in \F$, where we can and do
take $t_1=t$, which are algebraically independent over 
$\F_{\mbox{alg}}$ and such that $\F$ is a finite extension of 
$\F_{\mbox{alg}}(t_1,\ldots ,t_u)$.

We set the field $\E$ to be $\F_{\mbox{alg}}$ in the characteristic zero
case and $\F_{\mbox{alg}}(t)$ in positive characteristic, and put a suitable
discrete absolute value $\ep$ on $\E$. For 
$\E=\F_{\mbox{alg}}(t)$ we can simply use $\ep(f(t)/g(t))=
e^{\mbox{deg}(f)-\mbox{deg}(g)}$ which is clearly
discrete and with $\ep(t)\neq 1$. In characteristic zero $\E=\F_{\mbox{alg}}$
is a number field and we can use any absolute value $\ep$ obtained from
a prime ideal, which is also discrete. In either case $\E$ is a global
field and thus we can extend $(\E,\ep)$ to its completion 
$(\widehat{\E},\ep)$ which is locally
compact and with $\ep$ still discrete.

Now $\F_{\mbox{alg}}(t_1,\ldots ,t_u)$ embeds in $\widehat{\E}$ 
as the latter field has infinite transcendence degree over $\E$
by countability reasons, thus we can now regard
$\ep$ as defined on 
$\F_{\mbox{alg}}(t_1,\ldots ,t_u)$. Moreover in the characteristic
zero case we can send $t=t_1$ to a transcendental element of $\widehat{\E}$
with $\ep(t)\neq 1$ (for instance if $\alpha\neq 0$ is algebraic with 
$\epsilon(\alpha)\neq 1$ but $\epsilon(t)=1$ then $\alpha t$ is transcendental
and $\epsilon(\alpha t)\neq 1$), thus we have $\ep(t)\neq 1$ in either event. 
Finally as $\widehat{\E}$ is complete and locally compact,
there exists a finite field extension
$\E'$ of $\widehat{\E}$ with $\ep$ extending to $\E'$ (still discretely) 
and $\F$ embedding in $\E'$, so $\ep$ is a well defined discrete absolute 
value on $\F$ with $\ep(t)\neq 1$.

In positive characteristic any non zero element of the algebraic extension
$\F_{\mbox{alg}}$ is a root of unity as $\F_{\mbox{alg}}$ is a finite field,
thus we have established the non distortion of $A$
in this case and we now assume we are in characteristic zero for the rest of 
the proof. But here the above 
argument also tells us that we are done unless we come across an element $a$
of our free generating set $a_1,\ldots ,a_m$ for $A$ with no diagonal
entry which is a transcendental number. However if $t$ is an algebraic
number, and hence here also an element of $\F_{\mbox{alg}}$, then
our argument above of the existence of a suitable 
discrete absolute value $\ep$ on $\F_{\mbox{alg}}$,
which is now a number field,
will also go through as long as we can obtain $\ep$
from evaluation at a prime ideal where 
$\ep(t)\neq 1$ and this will be discrete
(whereupon we do not take $t_1=t$ anymore in the above,
rather $t_1$ becomes merely the first element in an
arbitrary transcendence basis). 

Thus we are only left with the case where every diagonal entry of $a$ is
algebraic and thus an element of the number field $\F_{\mbox{alg}}$, but
such that evaluation of each diagonal entry at any prime ideal in
$\F_{\mbox{alg}}$ equals 1. Let $\cal O$ be the ring of integers of
$\F_{\mbox{alg}}$, so that we have unique factorisation of a fractional
ideal of $\cal O$ into prime ideals, with the powers appearing in the
factorisation given by evaluation at these prime ideals. This extends to 
elements $x\in\F_{\mbox{alg}}^*$ by factorising the fractional ideal
$x\cal O$. Thus if $x$ is any diagonal entry of $a$ then it must now
factor trivially, giving us $x\cal O=\cal O$ and thus both $x$ and
$x^{-1}$ are in $\cal O$ so that $x$ is a unit of $\F_{\mbox{alg}}$.
This means that we can use Dirichlet's units theorem, which states
that if $(r_1,r_2)$ is the signature of $\F_{\mbox{alg}}$ (where $r_1$
is the number of real embeddings of $\F_{\mbox{alg}}$ and $r_2$ is the
number of pairs of complex conjugate embeddings then the group of units
${\cal U}(\F_{\mbox{alg}})$ is a finitely generated abelian group of $\Z$-rank
$r_1+r_2-1$ with torsion group equal to the roots of unity of 
$\F_{\mbox{alg}}$. Moreover the theorem also states that if the real
embeddings are $\sigma_1,\ldots ,\sigma_{r_1}$ and $\sigma_{r_1+1},
\ldots ,\sigma_{r_1+r_2}$ is a choice of complex embeddings from each
complex conjugate pair then the map
\[L(x)=\big(\log{|\sigma_1(x)|},\ldots , \log{|\sigma_{r_1}(x)|},
2 \log{|\sigma_{r_1+1}(x)|},\ldots , 2 \log{|\sigma_{r_1+r_2}(x)|},\big )\]
is an abelian group homomorphism from 
${\cal U}(\F_{\mbox{alg}})$ to the subspace
\[S=\{(\lambda_1,\ldots ,\lambda_{r_1+r_2})\in\R^{r_1+r_2}:
\lambda_1+\ldots +\lambda_{r_1+r_2}=0\}\]
of $\R^{r_1+r_2}$ with image of dimension $r_1+r_2-1$ which is a discrete
subgroup.

We can now finish off as follows: if we come across a generator $a$ of $A$
all of whose diagonal entries are units in $\F_{\mbox{alg}}$, we must have some
diagonal entry $d\in\F_{\mbox{alg}}^*$ with $L(d)\neq 0$, or
else all these diagonal entries are roots of unity and so $a$ has
a power which is unipotent. Now given any field embedding 
$\sigma_i$ from $\F_{\mbox{alg}}$
to $\R$ or $\C$, we can extend this to $\F$ by sending the elements
$t_1,\ldots ,t_u$ of our transcendental basis to 
algebraically independent transcendentals
in $\R$ or $\C$, then defining $\sigma_i$ over the finite extension
up to all of $\F$. Consequently we have that $L$ extends to a map from
$\F$ to $\R^{r_1+r_2}$ which is still an abelian group homomorphism.

We now write our abelian subgroup $A$ as a direct sum $A_u\oplus A_v$
where the generators $a_1,\ldots ,a_u$ of $A_u$ have all diagonal entries
which are units, but such that the generators $a_{u+1},\ldots ,a_m$ of
$A_v$ do not have this property. We then proceed as before on $A_v$
with coordinate functions to obtain the group homomorphism
$\theta:A\mapsto\R^{(m-u)d}$ which embeds $A_v$ discretely, so we have
$k_1>0$ such that $||\theta(a)||\geq k_1$ for $a\in A_v\setminus\{id\}$
but $\theta(A_u)=0$ which means that $||\theta(a)||\geq k_1$ for any
$a$ in $A\setminus A_u$. We also take the map
$M:A\rightarrow\R^{(r_1+r_2)d}$ defined by the $j$ coordinate
functions $M_j$ for $1\leq j\leq d$ where $M_j$ applies the 
Dirichlet map $L$
to the $j$th diagonal entry $a_{jj}$ of $a\in A$. Although $A_v$ could be mapped
in a highly non discrete way by $M$, we have that $M$ sends $A_u$ into
a discrete subset of $S^d\leq \R^{(r_1+r_2)d}$ and thus there exists
$k_2>0$ such that $||M(a)||\geq k_2$ for $a\in A_u\setminus\{id\}$. Thus
we can put $M$ and $\theta$ together to obtain the map
\[M\times\theta:A\rightarrow \R^{(r_1+r_2)d}\times \R^{(m-u)d}\]
where we put the max norm on the product. Consequently  for any non identity
element $a\in A$ we have \[||(M\times\theta)(a)||=\mbox{max}(||M(a)||,
||\theta(a)||)\geq\mbox{min}(k_1,k_2)>0.\]
Thus $A$ embeds discretely under $M\times\theta$, giving us $c>0$
such that
\[||(M\times\theta)(a_1^{n_1}\ldots a_m^{n_m})||\geq
c(|n_1|+\ldots +|n_m|).\]

Finally we apply Proposition \ref{subadd} where our list of absolute
values on $\F$ is not just $\epsilon_1,\ldots ,\epsilon_{m-u}$ obtained
by applying Lemma \ref{inj} to $A_v$, but also
$|\sigma_1|,\ldots,|\sigma_{r_1+r_2}|$ obtained from the embeddings
of $\F$ into $\R$ and $\C$. This results in our subadditive function
$f:GL(d,\F)\rightarrow[0,\infty)$. Moreover for $a\in A$ we have that
each coefficient of $(M\times\theta)(a)$ is obtained by applying
one of our absolute values to a particular diagonal entry $a_{jj}$
(hence also an eigenvalue) of $a$ and then taking logs. As we have
$f(a)\geq|\log{\epsilon_i(a_{jj})|}$ for all $i$ and $j$ by
Proposition \ref{subadd}, we conclude that $f(a)\geq
||(M\times\theta)(a)||$. Thus $f$ does have both properties and
therefore $A$ is undistorted in $G$.
\end{proof}   

\section{Some applications}
\subsection{Mapping class groups}
Of course any finitely generated group $G$ which fails to
satisfy any of the four conditions in Theorem \ref{know2}
cannot act properly  and semisimply by isometries on a CAT(0) space.
As we have now shown that $G$ also cannot be NIU-linear, we can
look to see if any of these obstructions have been established
for well known groups in the course of showing that they
do not admit nice actions on CAT(0) spaces.   
An important example of this is the mapping
class group $Mod(\Sigma_g)$ where here
$\Sigma_g$ will be an orientable surface of finite topological type having
genus $g$ at least 3 (which might be closed or might have any
number of punctures or boundary components). 
In \cite{bri} Bridson shows that for all the
surfaces $\Sigma_g$ mentioned above, the mapping class group $Mod(\Sigma_g)$
does not admit a proper and semisimple action on a complete CAT(0) space,
a result first credited to \cite{kaplb}. This is done using the
following obstruction which is similar to Theorem \ref{know2} Part (iv).

\begin{prop} \label{bri} (\cite{bri} Proposition 4.2)\\
If $\Sigma$ is an orientable surface of finite type having genus at
least 3 (with any number of boundary components and punctures) and if $T$
is the Dehn twist about any simple closed curve in $\Sigma$ then the
abelianisation of the centraliser in $Mod(\Sigma)$ of $T$ is finite.
\end{prop}

As this is covered by Theorem \ref{det}, we have
\begin{co} If $\Sigma$ is an orientable surface of finite type having genus at
least 3 (with any number of boundary components and punctures) then
$Mod(\Sigma)$ is not NIU-linear.
\end{co}
\begin{proof} We can combine Proposition \ref{bri} and Theorem \ref{det}
with $A$ the cyclic subgroup generated by the Dehn twist $T$ and $G$ the
centraliser of $A$ in the mapping class group. Thus our Dehn twist
must lie in the kernel of $\pi$ which is a torsion group if $Mod(\Sigma)$
and thus $G$ were NIU-linear,
but Dehn twists have infinite order.
\end{proof}

This immediately gives us:
\begin{co} \label{yoh}
If $\Sigma$ is an orientable surface of finite type having genus at
least 3 (with any number of boundary components and punctures)
and $\rho:Mod(S)\rightarrow GL(d,\F)$ is any linear representation
of the mapping class group of $S$ in any dimension $d$
where either $\F$ has positive
characteristic or $\F=\C$ and the image of $\rho$ lies in the
unitary group $U(d)$ then $\rho(T)$ has finite order for $T$ any
Dehn twist.
\end{co}

We note that there are ``quantum'' linear representations with
infinite image but where every Dehn twist has finite order.
However 
linearity of the mapping class group in genus $g\geq 3$
is a longstanding open question, although in genus 2 linearity over
$\C$ was established in \cite{bbud} and in \cite{kork}
by applying results on braid groups. 

\subsection{3-manifold groups}
Here we will restrict our 3-manifold groups to $\pi_1(M^3)$
for $M^3$ any closed 3-manifold. We note that linearity of
$\pi_1(M^3)$ is still unknown for some graph manifolds $M^3$,
but from the work of Agol - Wise and other related results
we have that $\pi_1(M^3)$ is virtually special (and hence
linear over $\Z$) if and only
if $\pi_1(M^3)$ is non positively curved, which means that
$M^3$ admits a smooth Riemannian metric where all sectional curvatures
at every point are bounded above by zero. 
We might therefore ask how 
NIU-linearity or VUF-linearity of $\pi_1(M^3)$ relates to the
non positive curvature of $M^3$.

Whilst we cannot answer
this in full here because of the open cases of graph manifolds,
we can at least do this quickly for 3-manifolds admitting a
geometric structure. Note that for very small fundamental
groups, namely those which are virtually cyclic and so have
geometry modelled on $S^3$ or $S^2\times\R$, we have a dichotomy
because the manifolds fail to be non positively curved but the
fundamental groups are even word hyperbolic. However apart from these examples
the correspondence works.
\begin{thm} \label{3m}
Let $M^3$ be a closed 3-manifold which admits one
of Thurston's eight model geometries and where $\pi_1(M^3)$
is not virtually cyclic. Then $\pi_1(M^3)$ is
NIU-linear if and only if it is VUF-linear if and only if
$M^3$ is non positively curved.
\end{thm}  
\begin{proof}
We will work in characteristic zero, so that the difference
between NIU and VUF-linearity disappears (the arguments do
all work in positive characteristic except for the hyperbolic
case, where it is not clear if we have faithful linear representations).
As $\pi_1(M^3)$ is not virtually cyclic, we have six geometries to check.

First if $M^3$ is hyperbolic then it has a faithful (and indeed
discrete) representation
in $PSL(2,\C)$ and also in $SL(2,\C)$ where every element is
loxodromic, thus without unipotents. Also if $M^3$ is Euclidean
then its fundamental group is virtually $\Z^3$ and all such groups
are NIU-linear (by commensurability and taking direct products).

If $M^3$ has Nil or Sol geometry then it cannot admit a metric of
non positive curvature and its fundamental group is virtually
polycyclic, but not virtually abelian so 
$\pi_1(M^3)$ is not NIU-linear by Section 4.

We are now left with the $\widetilde{PSL(2,\R)}$ and $\HH^2\times\R$
geometries, leading us to Seifert fibred spaces which here can be defined
as closed 3-manifolds $M^3$ which are finitely covered by circle bundles
over surfaces $B^3$ and where $\pi_1(M^3)$ is not virtually cyclic.
By say (C.10) and (C.11) of \cite{afw3m}, we have two
possibilities for $B^3$. The first is that it
has a finite cover which is of the form $S^1\times F$ for $F$ a closed
orientable surface, whereupon $M^3$ is non positively curved.
Thus again by commensurability and direct products,
the fundamental group of this finite cover, and hence of $M^3$ itself, is
NIU-linear.

The other case, where $\pi_1(M^3)$ will fail to be non positively curved,
is when $B^3$ has non zero Euler class, in which case the infinite order
central element of $\pi_1(B^3)$ has 
finite order in the abelianisation of $\pi_1(B^3)$ so Theorem
\ref{det} immediately applies to tell us that
$\pi_1(B^3)$ and $\pi_1(M^3)$ cannot be NIU-linear.
\end{proof}

\subsection{Arbitrary groups of Euclidean isometries}
A standard result is that the group ${\cal E}_d$ of Euclidean isometries
of $\R^d$ in dimension $d$ is a semidirect product of the form
$T_d\rtimes O(d)$, where the translation subgroup $T_d$ is a copy of
$(\R^d,+)$. As $O(d)$ is VUF-linear and hence shares
the nice properties of non positively curved groups as shown
in the previous sections, we can ask whether ${\cal E}_d$ and its
subgroups do too. This seems likely given that ${\cal E}_d$ is ``nearly''
VUF-linear in that we just have to deal with the translation subgroup.
In fact we will see that the abelianisation of centralisers property
still holds and the property of small subgroups almost goes through,
whereas things are rather different for distortion of abelian
subgroups. We emphasise here that we are interested in arbitrary groups
of Euclidean isometries in that we will not assume they are discrete, proper 
or cocompact.

First ${\cal E}_d$ is itself well known to be linear
as a subgroup of $GL(d+1,\R)$ via sending the map $Mx+b$ to the matrix
\[\sma{cc}\boxed{M}&b\\
0&1\fma\mbox{ for }M\in O(d)\mbox{ and }b\in\R^d.
\]
Of course under this embedding the translations map to unipotent elements.
In fact ${\cal E}_d$ is not NIU-linear as an abstract group either, even
for $d=2$ because the subgroup generated by a translation and an infinite
order rotation is (again) $\Z\,\wr\,\Z$. Now this group is not NIU-linear
(for instance by Corollary \ref{cos} as it is torsion free).

We first show the equivalent of Corollary \ref{coiv} for Euclidean isometries.
\begin{co} \label{cene}
If $G$ is any finitely generated group of Euclidean isometries in dimension
$d$ and 
$A\cong\Z^m$ is central in $G$ then there exists a subgroup
of finite index in $G$ that contains $A$ as a direct factor.
\end{co}
\begin{proof} From before it is enough to find a homomorphism from
$G$ to an abelian group which is injective on $A$. 
First consider the translation subgroup $S=T_d\cap A\cong\Z^r$ (for
$0\leq r\leq m$) of $A\cong\Z^m$ which must be a direct factor
of $B=S\oplus R$, where $B$ has finite index in $A$ and 
$R$ is some complementary subgroup. We will henceforth work with $B$ and
replace it by $A$ at the end.

Now the homomorphism $\psi:G\rightarrow O(d)$
which sends an isometry $Mx+b$ to $M$ is
injective on $R$ with $\psi(R)$ VUF-linear and torsion free, so we
can further compose with $\theta:\psi(G)\rightarrow(\C^*)^k$ as in Theorem
\ref{det}. Then $\theta\psi$ is injective on $R$ and trivial on the
translations $S$ in $A$, so we now need a homomorphism 
$\chi$ from $G$ to an abelian group which is injective on $S$. To obtain
this, first note that $\gamma\in{\cal E}_d$ commutes with the translation
$x\mapsto x+t$ if and only if $M_{\gamma}$ fixes the vector $t\in\R^d$, where
we write $\gamma(x)=M_{\gamma}x+b_{\gamma}$. Thus if the 
subgroup $S$ is freely generated
by the translations with vectors $t_1,\ldots ,t_r$ then for all
$g\in G$ we see that $M_g$ fixes the subspace 
$U=span(t_1,\ldots ,t_r)\leq \R^d$ pointwise.

Now on writing $\R^d=U\oplus U^{\perp}$ and letting
$\pi:\R^d\rightarrow U$ be orthogonal projection onto $U$, it is
easily checked that sending the element $g\in G$ to the map
$\pi\circ g|_U$ from $U$ to itself results in the translation
$u\mapsto u+\pi(b_g)$ and this is a group homomorphism $h$ from $G$
to the (abelian) translation group of $U$. This certainly sends
the translations $s \in S$ to (restrictions of) themselves, so
is injective on $S$.

Finally, putting $h$ and $\theta\psi$ together results in a homomorphism
from $G$ to an abelian group which is injective on $B$ and therefore
on $A$, as $B$ has finite index in $A$ which is torsion free.
\end{proof}

Any result on small subgroups of Euclidean groups has to take account
of the fact that wreath products can appear. The following is a
corollary of the results in Section 4.
\begin{co} \label{smae}
Suppose that $S$ is an arbitrary group of Euclidean isometries in any
dimension which does not contain a non abelian free group. Then $S$
is virtually metabelian.

Suppose further that $G$ is an arbitrary finitely generated group
of Euclidean isometries and and $S$ is any subgroup of $G$ which
does not contain a non abelian free group. If the translation
subgroup of $S$ is finitely generated
then $S$ is finitely generated, though the converse need
not hold. 
\end{co}
\begin{proof}
By the Tits alternative in zero characteristic and Theorem \ref{alt},
we have that $S$ is solvable without loss of generality and $\psi(S)$
is VUF-linear so is virtually abelian, meaning that $S$ has a finite
index subgroup $H$ with $\psi(H)$ abelian but the kernel of $\psi$
contains only translations.

Then by Corollary \ref{sbgp} applied to $\psi(G)$ and $\psi(S)$, we
see that $\psi(S)$ is finitely generated, so $S$ is an extension
of two finitely generated groups if its translation subgroup is also
finitely generated. However for $\Z\,\wr\,\Z$ in dimension
2 in our example above, the translation subgroup is not finitely generated. 
\end{proof}

Finally we show as an immediate corollary of Theorem \ref{undis}
that infinite cyclic subgroups of any finitely
generated group of Euclidean isometries are undistorted. This may not seem
surprising, until we combine it with \cite{conir} Lemma 3.5
which states that if $\phi\in GL(n,\Z)$ for $n\geq 2$ has
irreducible characteristic polynomial over $\Z$ and some
eigenvalue on the unit circle then the group
$G=\Z^n\rtimes_\phi\Z$ embeds
in ${\cal E}_3$ (with the $\Z^n$ as translations and $\Z$ a skew
translation). But if $\phi$ has some other eigenvalue not on
the unit circle then $\Z^n$ is distorted in $G$, just as in Example
1.5. Thus in Euclidean groups the geometric behaviour of $\Z$ subgroups
and higher rank free abelian subgroups can be quite different.  
\begin{co} \label{undse} 
If $G$ is any finitely generated group of Euclidean isometries and
$\gamma\in G$ is any infinite order element then 
$\langle\gamma\rangle$ is undistorted in $G$.
\end{co}
\begin{proof} If the rotation part $M_{\gamma}$ of $\gamma$ has infinite order
then $\langle \gamma\rangle$ distorted in $G$ means that 
$\langle \psi(\gamma)\rangle$
is distorted in the NIU-linear group $\psi(G)\leq O(d)$ which contradicts
Theorem \ref{undis}. Thus we can assume that some power of $\gamma$, so
$\gamma$ itself without loss of generality, is a translation. 
However translations are undistorted because they are translations: a
fact that has been recognised in many contexts (for instance \cite{lmr}
Lemma 2.8). In slightly more detail: as $G$ acts on a metric space $X$
by isometries, on taking an arbitrary basepoint $x_0\in X$ we obtain 
the displacement function $\delta_{x_0}(g)$ which measures
the distance in $X$ from $x_0$ to $g(x_0)$ and this is submultiplicative
on $G$ because it acts by isometries. 
Consequently finite generation says
we have $K>0$ such that 
$\delta_{x_0}(g)\leq Kl_S(g)$ for all $g\in G$ (where $l_S$ is the word
length on $G$ under a finite generating set $S$). Now if $\gamma$ is any
element of $G$ having a point $x_0\in X$ and a constant $c>0$ such that 
$cn\leq\delta_{x_0}(\gamma^n)$ for all $n\in\N$ then $\langle\gamma\rangle$
is undistorted in $G$ because $c/K>0$.
\end{proof}

\subsection{$Out(F_n$) and related results}
We finish by discussing the case of the 
outer automorphism group $Out(F_n)$ of the free group, as well as
some related groups. Now, at least for $n\geq 4$, 
$Out(F_n)$ is not linear over any field by \cite{fp} so 
we cannot apply any of our results directly to $Out(F_n)$.
However the Tits alternative was shown to hold
by \cite{bfh1} and \cite{bfh2}. Then
\cite{alib} (which
ostensibly shows that all cyclic subgroups of $Out(F_n)$ are undistorted)
is able to obtain the strong small subgroups property for $Out(F_n)$ using
a result of G.\,Conner in \cite{consol}. 

This was described as follows: if $G$ is finitely generated then
put the word length $l_S$ on $G$ (with respect to
some finite generating set $S$) and let $\tau$ be the associated translation
length, that is $\tau(g)=\mbox{lim}_{n\rightarrow\infty}l_S(g^n)/n$.
Thus having $\tau(g)>0$ for all infinite order elements $g$ is equivalent
to saying that every cyclic subgroup of $G$ is undistorted. This is
not equivalent to saying that there exists $c>0$ with $\tau(g)\geq c$ for
all infinite order
$g\in G$, as seen by examples of the form $\Z^n\rtimes\Z$ where
all cyclic subgroups are undistorted but $\Z^n$ is distorted.

However now suppose the finitely generated group $G$ has finite
virtual cohomological dimension. Then  \cite{consol} Theorem 3.4 states
that if $G$ does possess such a constant $c>0$ as above then every
solvable subgroup of $G$ is virtually abelian and finitely generated.
Thus if also the classical Tits alternative
is known to hold for $G$, namely every subgroup either contains $F_2$ or
is virtually solvable, we obtain the strong small subgroups property for
$G$. As Alibegovic actually shows in \cite{alib} that there is indeed such
a $c_n>0$ for $Out(F_n)$ with a suitable finite generating set, we are
done. (We remark though that this approach still requires 
establishing the classical Tits alternative
for $Out(F_n)$ as in \cite{bfh1} and \cite{bfh2}, which is highly non trivial, 
as well as relying on the
finite virtual cohomological dimension
of $Out(F_n)$ from \cite{cv}.)

But what about abelian subgroups of $Out(F_n)$? The recent preprint
\cite{wrig} shows, by  employing
technical knowledge of train track maps,
that they are all undistorted. However here we can use the above result
of Alibegovic to provide a quick proof of this fact.
\begin{thm} \label{trids}
Let $G$ be a finitely generated group where the translation length
function $\tau:G\rightarrow [0,\infty)$
with respect to word length of some (equivalently any)
finite generating set 
has $c>0$ such that $\tau(g)\geq c$ for all infinite order elements
$g\in G$. Then any finitely generated abelian subgroup $A$ of $G$ is
undistorted in $G$.
\end{thm}
\begin{proof}
As usual, without loss of generality we take $A$ to be isomorphic
to $\Z^m$ and pick out some free abelian basis $a_1,\ldots ,a_m$.
Here by an $\R$-norm $||\cdot ||$on $\R^m$ we mean the standard definition from
normed vector spaces, that is it satisfies the triangle inequality with
$||v||$ being zero if and only if $v$ is the zero vector, and also
$||\lambda v||=|\lambda|\cdot ||v||$ for $|\cdot |$ the usual modulus
on $\R$. We will also define a $\Z$-norm on $\Z^m$ to be a function
$f:\Z^m\rightarrow [0,\infty)$ having the same properties, except
the last becomes $f(na)=|n|\cdot||a||$ for all $n\in\Z$ and $a\in\Z^m$.
We also have $\R$- and $\Z$-seminorms where we remove the $||\cdot ||=0$
implies $\cdot=0$ condition.

Now consider word length $l_S$ on $G$ with respect to some finite
generating set $S$ and the associated translation length $\tau$
(where changing the generating set replaces both
$l_S$ and $\tau$ by Lipschitz equivalent functions). We have
$0\leq\tau(g)\leq l_S(g)$ for any $g\in G$ by repeated use of
the triangle inequality and $\tau$ also satisfies 
$\tau(g^n)=|n|\tau(g)$. We further have $\tau(gh)\leq\tau(g)+\tau(h)$
for commuting elements $g,h$ (but not in general). Thus on
restricting $\tau$ to the abelian subgroup $A$, we see that $\tau$ is
 a $\Z$-seminorm on $A$. However here it is also a $\Z$-norm as
$\tau(a)>0$ for all $a\in A\setminus\{id\}$. Moreover $\tau$ is actually
a discrete $\Z$-norm in that we have $\tau(a)\geq c>0$ for 
all $a\in A\setminus\{id\}$.

On regarding $A\cong\Z^m$ as embedded in $\R^m$ as the integer lattice
points, we can extend $\tau$ to $\Q^m$ by dividing through and to
$\R^m$ by taking limits, so that $\tau$ is also an $\R$-seminorm
on $\R^m$. However it is not too hard to see that $\tau$ is in fact
a genuine $\R$-norm, as explained carefully in \cite{step}, and we
denote this by $||\cdot ||_\tau$.

Now take $a=a_1^{n_1}\ldots a_m^{n_m}\in A$ which is also the lattice
point $(n_1,\ldots ,n_m)\in\R^m$. We have the $\ell_1$ norm 
$||\cdot ||_1$ on $\R^m$
but all norms on $\R^m$ are equivalent, so there is $k>0$ such that
\[l_S(a)\geq\tau(a)=||(n_1,\ldots ,n_m)||_\tau
\geq k||(n_1,\ldots ,n_m)||_1=k(|n_1|+\ldots +|n_m|)\]
so $A$ is undistorted in $G$.
\end{proof}

We end by pointing out that $Out(F_n)$ having
undistorted abelian subgroups can be used to
establish some further consequences.       
\begin{co} If $A$ is any abelian subgroup of the automorphism
group $Aut(F_n)$, or of any free by cyclic group $G=F_n\rtimes_\alpha\Z$
for $\alpha\in Aut(F_n)$ then $A$ is finitely generated and
undistorted in $Aut(F_n)$ or in $G$.
\end{co}
\begin{proof}
If $S,H,G$ are all finitely generated groups with $S\leq H\leq G$ and
$S$ is undistorted in $G$
then $S$ is undistorted in $H$ (else extend the generating set of $H$ to
one of $G$, whereupon the distortion persists). The converse also holds
if $H$ has finite index in $G$.
Now an observation dating
back to Magnus is that $Aut(F_n)$ embeds in $Out(F_{n+1})$ by
considering automorphisms of $F_{n+1}$ which fix the last element of the
basis.

Moreover for the free by cyclic group $G$, we have that if $\alpha$
has infinite order in $Out(F_n)$ then $G$ embeds in $Aut(F_n)$. This
can be seen on taking the copy $F_n$ of inner automorphisms in $Aut(F_n)$
and then $G$ is isomorphic to $\langle \alpha,F_n\rangle\leq Aut(F_n)$.
If however $\alpha$ has finite order then $G$ contains the finite index
subgroup $H=F_n\times\Z$ and this certainly has finitely generated and
undistorted abelian subgroups.
\end{proof}

In \cite{hgwsp} various non hyperbolic
free by cyclic groups are shown to
act freely on a CAT(0) cube complex, though this complex might not
be finite dimensional or locally finite. In particular the action
need not be cocompact (unlike the hyperbolic case where they
had already established the existence of geometric actions).
It is pointed out that this implies such groups have undistorted
abelian subgroups, by using axis theorems of Haglund and of Woodhouse
(since here acting freely implies acting properly,
because any compact set contains only finitely many vertices).
From the result above, we have a unified proof for all free by
cyclic groups (though of course this fact is immediate if the group
is word hyperbolic).

Free by cyclic groups are good test cases for NIU-linearity, as they are for
non positive curvature, because their restricted subgroup structure
means that they satisfy all properties in Theorem \ref{know2} (or
indeed in Theorem \ref{know}). In \cite{meoth} we look in more
detail as to which well known abstract groups are NIU-linear. In
particular we show that the famous Gersten free by cyclic group
$F_3\rtimes\Z$ is not NIU-linear. This accords with the fact that
it does not act properly and semisimply on any complete CAT(0)
space, even though it passes all the obstructions we have seen
for non positive curvature.

\Address

\end{document}